\documentclass[11pt]{article}
\usepackage{amsmath,amssymb,amsthm,graphicx,fixmath}

\usepackage{algpseudocode}
\usepackage{algorithm}
\usepackage{setspace}
\usepackage{ifthen}
\usepackage{tikz}

\renewcommand{\paragraph}[1]{\par\medskip\noindent\textbf{#1}\ }

\def\Re{\mathbb R}

\providecommand{\remove}[1]{}

\theoremstyle{plain}
\newtheorem{theorem}{Theorem}[section]
\newtheorem{lemma}[theorem]{Lemma}

\newtheorem{proposition}[theorem]{Proposition}
\newtheorem{claim}[theorem]{Claim}
\newtheorem{corollary}[theorem]{Corollary}

\newtheorem{problem}[theorem]{Problem}

\newtheorem{observation}[theorem]{Observation}
\theoremstyle{definition}
\newtheorem{definition}[theorem]{Definition}
\theoremstyle{remark}

\newcommand{\G}{\mathcal{G}}

\newcommand{\E}{\mathcal{E}}

\renewcommand{\P}{\mathcal{P}}

\newcommand{\F}{\mathcal{F}}

\def\famG{\mathcal{G}}

\def\famH{\mathcal{H}}

\newcommand{\card}[1]{\lvert {#1} \rvert}

\makeatletter
\renewcommand*\@fnsymbol[1]{\ensuremath{%
  \ifcase#1\or
    *\or
    \star\or     % <-- זה הסימון של ה-thanks השני (במקום †)
    \ddagger\or
    \mathsection\or
    \mathparagraph\or
    \|\or
    **\or
    \star\star\or
    \ddagger\ddagger
  \else
    \@ctrerr
  \fi}}
\makeatother

\begin{document}
	
	\title{New Sufficient Conditions for Linear-Sized Epsilon-Nets and $(p,2)$-Theorems
    %Linear epsilon-nets and $(p,2)$-theorems in wider settings
    }
\author{Chaya Keller\thanks{School of Computer Science, Ariel University, Israel. \texttt{chayak@ariel.ac.il}. Research partially supported by the Israel Science Foundation (grant no. 1065/20).}
	\mbox{ }
	and Shakhar Smorodinsky\thanks{Department of Computer Science, Ben-Gurion University of the NEGEV, Be’er Sheva 84105, Israel. Research partially supported by the Israel Science Foundation (grant no. 1065/20). 
		\texttt{shakhar@bgu.ac.il}
}}

	\date{}
	\maketitle

	\begin{abstract}
An $\epsilon$-net theorem for a hypergraph upper bounds the minimum size of a vertex set that pierces all $\epsilon$-heavy hyperedges. A $(p,2)$-theorem bounds from above the minimum size of a vertex set that pierces all hyperedges, in terms of the maximum size of a set of pairwise disjoint hyperedges.
Numerous works studied $\epsilon$-net theorems and $(p,2)$-theorems that guarantee the existence of small-sized piercing sets.

We focus on the question: In which settings the asymptotically smallest possible piercing
sets --- i.e., $\epsilon$-nets of size
$O(\frac{1}{\epsilon})$ and piercing sets of size $O(p)$ in $(p,2)$-theorems, are guaranteed? We obtain several sufficient criteria for the existence
of such \emph{linear} $\epsilon$-net theorems and $(p,2)$-theorems that unveil interesting connections
to graph theory and improve and generalize several previous results. Most notably, we exhibit an 
unexpected relation of $\epsilon$-nets to the classical \emph{Zarankiewicz's problem} in graph theory. We show that a linear bound in the Zarankiewicz-type problem that asks for the maximum size of a bipartite graph with no copy of $K_{2,t}$, implies a linear $\epsilon$-net theorem for the corresponding neighborhood hypergraph. We also show that hypergraphs with a hereditarily linear-sized \emph{Delaunay graph} admit an almost linear $(p,2)$-theorem, and deduce that incidence hypergraphs of non-piercing regions in the plane admit a linear $(p,2)$-theorem, significantly improving previous results on such hypergraphs.  

Our work presents a landscape of sufficient conditions for the existence of linear $\epsilon$-net theorems and $(p,2)$-theorems, with complex interrelations between them. Many of the 
interrelations are still unknown and call for future research.  
%In this paper we focus on $(p,2)$-theorems, $\varepsilon$-nets and the so-called Zarankiewicz problem in extremal graph theory. We reveal a connection between Zarankiewicz-type bounds, linear $\varepsilon$-nets, and linear or near-linear $(p,2)$-theorems. By making these relationships explicit, we obtain simplified proofs of known results, significant improvements in several classical settings (including non-piercing regions), and new theorems. 
%In particular, we prove that non-piercing regions admits a $(p,2)$-theorem with an optimal linear transversal $O(p)$ improving on the best previously known bound of $O(p^9)$. 
	\end{abstract}

\section{Introduction}

\subsection{Background}

A hypergraph $H=(V,\E)$ consists of a set $V$ of vertices and a set $\E$ of subsets of $V$, called hyperedges. In this paper we focus on hypergraphs encountered in geometric instances, which we call in the sequel \emph{geometric hypergraphs}. A common structure of such a hypergraph $H=(V,\E)$ is where $V \subset \mathbb{R}^d$ and $\E$ encodes incidences between the points in $V$ and a family $\mathcal{S}$ of well-behaved shapes (e.g., halfspaces or axis-parallel boxes). That is, $\E=\{V \cap S: S \in \mathcal{S}\}$. %Where $V=\mathbb{R}^d$, we call such a hypergraph `the incidence hypergraph of points vs.~the shapes in $\mathcal{S}$. 

A \emph{hitting set} (or a \emph{transversal}) is a set of vertices that pierces all hyperedges of the hypergraph -- i.e., $S \subset V$ such that for each $e \in \E$, $e \cap S \neq \emptyset$. The \emph{transversal number} of a hypergraph is the size of the smallest transversal.
A vertex set $S \subset V$ is \emph{shattered} by $\E$ if $\{e \cap S:e \in \E\} = 2^{S}$. The \emph{VC-dimension} of $H$ is the maximum size of a shattered set of vertices. Roughly speaking, most classes of geometric hypergraphs have a bounded VC-dimension (see~\cite{Matousek2002}).

\medskip \noindent \textbf{$\epsilon$-nets.} 
An \emph{$\epsilon$-net} for a hypergraph $H=(V,\E)$ with $|V|=n$ is a set of vertices that pierces every hyperedge of $H$ of size at least $\epsilon n$. The systematic study of $\epsilon$-nets was initiated by Haussler and Welzl \cite{HausslerW87}, who proved that any hypergraph with VC-dimension $d$ admits an $\epsilon$-net of size $O(\frac{d}{\epsilon} \log\frac{d}{\epsilon})$, and consequently, most geometric hypergraphs have small-sized $\epsilon$-nets.
%later improved to $O(\frac{d}{\epsilon} \log\frac{1}{\epsilon})$. 
In the last three decades, $\epsilon$-nets have become a central tool in diverse areas of computer science, including machine learning (e.g.,~\cite{Blumer1989}), algorithms (e.g.,~\cite{T.M.Chan2018}) and computational geometry (e.g.,~\cite{Arya2018}).

\medskip \noindent \textbf{Linear $\epsilon$-net theorems.} While the upper bound of Haussler and Welzl~\cite{HausslerW87} on the guaranteed size of an $\epsilon$-net is known to be essentially tight in general even for geometric hypergraphs (see~\cite{Komlos1992b,PachT13}), for various classes of geometric hypergraphs the guaranteed size of an $\epsilon$-net is much smaller, and in many cases even $\epsilon$-nets of size $O(\frac{1}{\epsilon})$ can be guaranteed. Besides their intrinsic interest, such improved bounds on the size of the $\epsilon$-net are useful for obtaining approximations for the \emph{hitting set} problem. Indeed, Br{\"{o}}nnimann and Goodrich~\cite{BronnimanG95} showed that an efficient construction of a \emph{linear-sized} $\epsilon$-net (i.e., an $\epsilon$-net of size $O(\frac{1}{\epsilon})$) implies a constant-factor approximation algorithm for the hitting set problem. As a result, numerous works (e.g.,~\cite{BusGMR16,ChanGKS12,MatousekSW90,MustafaDG18,PyrgaR08}) 
%\cite{Alon12,ClarksonV07,KupavskiiMP16}
studied different aspects of 
the following general problem.
\begin{problem}\label{Q:Small-eps-net}
Determine all classes of hypergraphs that admit a linear-sized $\epsilon$-net.  
\end{problem}
Some of these works (e.g.,~\cite{ChanGKS12,MustafaDG18,PyrgaR08,Varadarajan09}) focused on sufficient conditions for having a linear-sized  $\epsilon$-net. The sufficient conditions included, e.g., having a \emph{linear support} and a \emph{linear shallow cell complexity} (see definitions below). 
%having a \emph{planar support} \textcolor{blue}{linear support}, a \emph{linear union complexity} \textcolor{blue}{This is a property of the underlying ground set in geometric settings}, and a linear \emph{shallow cell complexity} (see definitions below). 

\medskip \noindent \textbf{$(p,q)$-theorems.} A \emph{$(p,q)$ theorem} for a class $\famH$ of hypergraphs asserts that for any $H=(V,\E)$ in $\mathcal{H}$, if among any $p$ hyperedges in $\E$ some $q$ have a non-empty intersection, then $H$ admits a `small' hitting set. The original $(p,q)$-theorem was obtained by Alon and Kleitman~\cite{Alon1992} for the class of hypergraphs whose vertex set is $\mathbb{R}^d$ and whose hyperedges are 
finite families of convex sets (or in other words, for families of convex sets in $\mathbb{R}^d$); the size of the guaranteed hitting set is $c=c(d,p,q)$. Being a generalization of the classical Helly's theorem (that asserts the same for $p=q=d+1$, guaranteeing a hitting set of size $1$), this theorem may look as manifesting a property of convex sets. However, the $(p,q)$-theorem was generalized to a wide variety of non-convex settings and has found applications to various fields. For example, Matou\v{s}ek~\cite{Matousek04} proved a $(p,q)$-theorem for hypergraphs with a bounded VC dimension that has major applications in machine learning, computational geometry, and model theory (see, e.g.,~\cite{BK22,Kaplan24,RolnickS17}).  

\medskip \noindent \textbf{Linear $(p,2)$-theorems.}
An especially important class of $(p,q)$ theorems is $(p,2)$ theorems, that bound the transversal number $\tau$ of the hypergraph in terms of its packing number $\nu$ (i.e., the maximum number of pairwise disjoint hyperedges). Such theorems address the general problem of characterizing classes $\mathcal{H}$ of hypergraphs for which $\tau$ can be bounded as a function of $\nu$  --- a central problem in hypergraph theory that  includes several well-known conjectures, like \emph{Ryser's conjecture} (see~\cite{Matousek2002}) and 
\emph{Tuza's conjecture}~\cite{Tuza81}, as special cases. $(p,2)$-theorems were obtained for various classes of geometric hypergraphs~\cite{Chan2012,DJ11,Kar91,KellerS20,KNPS06}.
%e.g., for the hypergraph whose vertex set is $\mathbb{R}^d$ and whose hyperedges are induced by all axis-parallel boxes in $\mathbb{R}^d$~\cite{Kar91}. 
Obviously, the size of a hitting set guaranteed in a $(p,2)$-theorem cannot be less than $p-1$, and hence, a \emph{linear-sized} transversal -- i.e., a transversal of size $O(p)$, is the best one can hope to obtain. This gives rise to the following general problem. 
\begin{problem}\label{Q:Linear-p2}
Determine all classes of hypergraphs that admit a linear-sized $(p,2)$-theorem.  
\end{problem}
A class of hypergraphs that admits such a strong $(p,2)$-theorem is all hypergraphs whose vertex set is $\mathbb{R}^2$ and whose hyperedges encode incidences of points and  a family of \emph{pseudo-disks}~\cite{Chan2012,Pinchasi15} (see exact definition below).

\medskip \noindent \textbf{Zarankiewicz's problem.}
Given $n,t \in \mathbb{N}$, \emph{Zarankiewicz's problem} asks: What is the maximum number of edges in a bipartite graph $G=(A \cup B,E)$ on $n$ vertices that does not contain a copy of the complete bipartite graph $K_{t,t}$? This question is one of the central open problems in extremal graph theory (see~\cite{Sudakov2010}). The best known general upper bound is $O(n^{2-\frac{1}{t}})$, proved by K{\H o}v\'{a}ri, S{\'o}s and Tur{\'a}n~\cite{Koevari1954} over 70 years ago. In the last years, numerous works obtained better upper bounds for specific classes of graphs. A well-known result of this type is by Fox et al.~\cite{FoxPSSZ17} who proved a bound of $O(n^{2-\frac{1}{d}})$ for graphs whose \emph{neighborhood hypergraph} (namely, a hypergraph whose vertex set is $A$ and whose hyperedge set is $\{\{a\in A: (a,b) \in E\}: b\in B\}$; see detailed definition below) has VC-dimension $d$. A series of recent works obtained even stronger bounds for graphs encoding incidences of points and various classes of geometric objects (e.g.,~\cite{BasitCSTT21,ChalermsookOZ25,ChanH25,HunterMTS25,KellerS24,Tomon2021}).  
The \emph{lopsided} Zarankiewicz's problem is the generalized variant where the forbidden configuration is $K_{s,t}$. Numerous works studied this variant (e.g.,~\cite{Bukh24,Conlon21,KollarRS96}), and  in~\cite{ChanKS25} results in this setting were used to obtain better bounds in the \emph{hypergraph Zarankiewicz problem}~\cite{Erdos64}.

\subsection{Our results} In this paper we study Problems~\ref{Q:Small-eps-net} and~\ref{Q:Linear-p2}. We obtain new sufficient criteria for the existence
of \emph{linear} $\epsilon$-net theorems and $(p,2)$-theorems that uncover interesting connections to graph theory, in particular to Zarankiewicz's problem, and improve and generalize several previous results. 

\subsubsection{Linear lopsided Zarankiewicz bound implies linear $\epsilon$-net theorem}

In order to present our results, we need a few more definitions. 

Given a bipartite graph $G=(A\cup B,E)$ (with vertex parts
%\footnote{\textcolor{red}{Along this paper we slightly abuse notation, and distinguish between $A$ and $B$, referring to $A$ as the left part and to $B$ as the right part.}}  
$A$ and $B$), the associated (primal) neighborhood hypergraph is the hypergraph $H_G$ whose vertex set is $A$, and for each $v \in B$, the set $N(v)=\{u \in A: \{u,v\}\in E \}$ is a hyperedge.
We consider classes of induced bipartite graphs $\famG=\{G[P\cup S] : P \subset \P,S \subset \mathcal{S} \}$ where $\P$ and $\mathcal{S}$ are given ground sets. The corresponding (primal) class of neighborhood hypergraphs is $\famH_{\famG}=\{H_G:G \in \famG\}$.
For $c \in \mathbb{N}$, a class $\G$ of bipartite graphs has VC-dimension $d$ if for every $G \in \G$, the VC-dimension of $H_{G}$ is at most $d$.

Given a hypergraph $H=(V,\E)$, the dual hypergraph $H^*$ is the hypergraph whose vertex set is $\E$ and each $v \in V$ defines a hyperedge that preserves the original incidences. Hence, for $G=(A\cup B,E)$, 
$H_G^*$ is the hypergraph whose vertex set is $B$, and for each $v \in A$, the set $N(v)=\{u \in B: \{v,u\}\in E\}$ is a hyperedge. The class $\famH_{\G}^*$ consists of the hypergraphs $\famH_{G}^*$ for $G \in \G$. Note that the property of having bounded VC-dimension is hereditary: If a class $\G$ of bipartite graphs has bounded VC-dimension $c$, then for every $G=(P\cup S,E) \in \G$ and for every induced bipartite subgraph $G'=G [{P' \cup S'}]$ where $P' \subset P, S' \subset S$, the VC-dimension of $H_{G'}$ is at most $c$.

\begin{definition}
		A class $\G$ of bipartite graphs admits a \emph{hereditarily linear lopsided Zarankiewicz theorem} if there exists $c=c(\G)$ such that for any $t_1,t_2 \in \mathbb{N}$, for every $G \in \G$, and for every finite induced graph $G[P \cup S]$, if $G[P \cup S] $ is $K_{t_1,t_2}$-free then $\card {E(G[P \cup S])} \leq c(t_2|P|+t_1|S|)$. 
\end{definition}

It was proved in \cite{ChanKS25} that the intersection graph of vertical and horizontal segments in the plane (i.e., the bipartite graph $G=(A \cup B,E)$ where $A$ is a family of horizontal segments, $B$ is a family of vertical segments, and for $a \in A, b \in B$, $(a,b) \in E$ if $a \cap b \neq \emptyset$) admits a hereditarily linear lopsided Zarankiewicz theorem. It seems plausible that the same holds for intersection graphs of $y$-monotone pseudo-disks and maybe even for any intersection graph of two families of pairwise disjoint curves (called strings); the techniques of~\cite{HunterMTS25} can be possibly used to show this. 

%In the following theorem we reveal an interesting and somewhat surprising connection between Zaranckiewicz’s theorem from extremal graph theory and $\epsilon$-nets.
%The following theorem relates admitting a hereditarily optimal lopsided Zarankiewicz theorem to admitting a linear $\epsilon$-net.

We show that somewhat surprisingly, a hereditarily linear lopsided Zarankiewicz theorem for a class of graphs implies a linear $\epsilon$-net theorem for the corresponding class of primal neighborhood hypergraphs.  
\begin{theorem}\label{thm:lopsided->epsnet}
	Let $c>0$ and $d \in \mathbb{N}$. Let $\G$ be a class of bipartite graphs with VC-dimension $d$ that admits a hereditarily linear lopsided Zarankiewicz theorem, with a constant $c(\mathcal{G})=c$. Then for every $G \in \G$, the hypergraph $H_G$ admits an $\epsilon$-net of size at most $\frac{c'}{\epsilon}$, where $c'$ depends only on $c,d$.
\end{theorem} 
In fact, our proof shows that it is sufficient to assume that the hereditarily linear lopsided Zarankiewicz theorem holds only \emph{with respect to containment of graphs of the form $K_{2,t}$ for all $t$}.

A few remarks are due. First, Mustafa et al.~\cite{MustafaDG18} showed that all linear $\epsilon$-net theorems known today hold in settings where the \emph{shallow cell complexity} of the hypergraph is $O(1)$. We do not know whether a hereditarily linear lopsided Zarankiewicz theorem for $\mathcal{G}$ implies that the shallow cell complexity of $H_G$ for all $G \in \G$ is $O(1)$ (see Problem 3 in Section \ref{sec:disc} below). If the answer is negative, then Theorem~\ref{thm:lopsided->epsnet} provides a new class of linear $\epsilon$-net theorems that is significantly different from all previously known classes. 

Second, the requirement that $\famG$ has a bounded VC-dimension can be omitted if the class $\famH^*_{\G}$ has a linear support (or even a hereditarily linear Delaunay graph), see\cite[Theorem 6]{AKP21}.
Third, in the opposite direction, the existence of a stronger variant of $\epsilon$-net called \emph{$\epsilon$-$t$-net} (see~\cite{AJKSY22}) was shown in~\cite{KellerS24} to imply upper bounds for Zarankiewicz's problem.

\subsubsection{Hereditarily linear Delaunay graph for the dual hypergraph implies an almost linear $(p,2)$-theorem for the primal}
The \emph{Delaunay graph} of a hypergraph $H=(V,\E)$ is the graph $G$ on the vertex set $V$ whose edges are $\{e \in \E: |e|=2\}$.
We say that a hypergraph $H=(V,\E)$ admits a hereditarily $c$-linear Delaunay graph for $c>0$, if for every induced subhypergraph $H{\mid}_{V'}$ ($V' \subset V$), the corresponding Delaunay graph $G_{V'}$ satisfies $|E(G_{V'})| < c \cdot |V'|$. Sometimes we simply say that the Delaunay graph is \emph{hereditarily linear}, and suppress the dependence on $c$. Similarly, we say that a class $\famH$ of hypergraphs admits a hereditarily $c$-linear Delaunay graph if each $H \in \famH$ has a hereditarily $c$-linear Delaunay graph.

We show that if a class of dual hypergraphs admits a hereditarily linear Delaunay graph, then the class of corresponding \emph{primal hypergraphs} admits a $(p,2)$-theorem with an almost linear transversal size.
\begin{theorem}\label{thm:2}
Let $c>0$. Let $\famH $ be a class of hypergraphs such that $\famH^*$ admits a hereditarily $c$-linear Delaunay graph. Then $\famH$ admits a $(p,2)$-theorem with a transversal of size $c'p \log p$, where $c'$ depends only on $c$. 
	%Let $\F$ be a family of $n$ regions in $\Re^d$ and let $H$ be the associated dual hypergraph. Assume that $H$ admits a hereditarily linear Delaunay graph. If $\F$ satisfies the $(p,2)$-property then $\F$ admits an $O(p \log p)$-sized transversal.
\end{theorem}

%If this dual hypergraph admits a linear-sized epsilon-net (see definition in Section \ref{sec:main}), 
If in addition, $\famH$ admits a linear-sized $\epsilon$-net, then the assertion of Theorem \ref{thm:2} can be strengthened into a linear $(p,2)$-theorem.
\begin{corollary}\label{cor:3}
Let $c,c'>0$. Let $\famH $ be a class of hypergraphs such that $\famH$ admits an $\epsilon$-net of size $\frac{c}{\epsilon}$ and $\famH^*$ admits a hereditarily $c'$-linear Delaunay graph. Then $\famH$ admits a $(p,2)$-theorem with a transversal of size $c''p$, where $c''$ depends only on $c,c'$. 
\end{corollary}
The main tool in the proof of Theorem~\ref{thm:2} is an asymptotically optimal \emph{fractional Helly theorem}~\cite{Katchalski1979} for hypergraphs with a hereditarily linear Delaunay graph, which we incorporate into the Alon--Kleitman~\cite{Alon1992} proof technique of the $(p,q)$-theorem.
	
Let $H=(V,\E)$ be a hypergraph. We say that two vertices $u,v \in V$ are \emph{friends} if there exists some $e \in \E$ such that $u,v \in e$.
\begin{proposition}[A fractional Helly theorem for hypergraphs with a hereditarily linear Delaunay graph]\label{thm:1}
	Let $\famH$ be a class of hypergraphs with a hereditarily $c$-linear Delaunay graph, and let $0 < \alpha <1$. Then there exists $0<\beta = \beta (\alpha,c)<1$ such that for each $H=(V,\E) \in \famH$, if $\alpha \binom{n}{2}$ of the pairs in $\binom{V}{2}$ are friends, then there exists $e \in \E$ such that $|e|>\beta n$. Moreover, $\beta \geq c'\alpha$, where $c'$ depends only on $c$.
\end{proposition}

\subsubsection{Pseudo-disks, non-piercing regions and linear support} In order to compare our results with previous works, we discuss their implications on several well-studied classes of geometric hypergraphs.  

\medskip \noindent \textbf{Pseudo-disks and non-piercing regions.} A family $\F$ of simple Jordan regions in $\Re^2$ is called \emph{a family of pseudo-disks} if any two boundaries of members of $\F$ intersect at most twice. A generalization of this notion was presented by Raman and Ray~\cite{Raman2020}: A family $\F$ of regions in $\Re^2$ is \emph{non-piercing} if for every $F,G \in \F$, $G \setminus F$ is connected. The intersection hypergraph of two families $A,B$ of non-piercing regions has vertex set $A$ and hyperedges $\{w \in B: w \cap v \neq \emptyset\}$ for all $v \in A$. Special cases are the primal and dual incidence hypergraphs of points and a finite family $F$ of non-piercing regions. In the primal incidence hypergraph, the vertex set is $\mathbb{R}^2$ and the hyperedges are the elements of $F$. In the dual incidence hypergraph, the vertex set is $F$ and the hyperedges are all subsets of $F$ that have a non-empty intersection. 

\medskip \noindent \textbf{Planar support and $\alpha$-linear support.} For $\alpha>0$, a hypergraph $H=(V,\E)$ admits an \emph{$\alpha$-linear support} if for any $V' \subset V$ there exists a graph $G=G_{V'}$ defined on the vertex set $V'$, such that $|E(G_{V'})| \leq \alpha |V'|$, and for each hyperedge $e \in \E$, the induced graph $G[e]$ (i.e., the subgraph restricted to the vertices of $e$) is connected. For a class $\famH$ of hypergraphs, we say that $\famH$ admits an $\alpha$-linear support if each $H \in \famH$ admits an $\alpha$-linear support.
A hypergraph $H=(V,\E)$ admits a \emph{planar support} if in the above definition, instead of the condition $|E(G_{V'})| \leq \alpha |V'|$, it is required that $G_{V'}$ is a planar graph. Clearly, a planar support is a $3$-linear support.

\medskip \noindent \textbf{Relation between the notions.}
It was proved in \cite{Raman2020} (and later reproved in \cite{Dalal24}) that 
%if $G=(A \cup B,E)$ is the bipartite intersection graph of two families $A,B$ of non-piercing regions (i.e., for $v \in A, w \in B$, $(v,w) \in V$ if $v \cap w \neq \emptyset$), then the neighborhood hypergraph $H_G$ 
the intersection hypergraph of two families of non-piercing regions
admits a planar (and hence, a 3-linear) support. In particular, this holds for 
%any bipartite intersection graph 
the intersection hypergraph of any two families of pseudo-disks. In addition, it is clear from the definition that if a class $\famH$ of hypergraphs admits a linear support, then it admits a hereditarily linear Delaunay graph. These relations allow comparing our results to several previous works.
%Hence, we have the following chains of containments:
%{\it 
%\[
%\mbox{Pseudo-disks} \subset \mbox{Non-piercing regions} \subset \mbox{Linear support} \subset \mbox{Hered. Linear Delaunay}
%\]
%}
%\textcolor{blue}{\[
%\begin{aligned}
%    &\text{intersection hyp. of two families of Pseudo-disks} \\
%    &\subset \text{intersection hyp. of two families of Non-piercing regions} \\
%    &\subset \text{hypergraphs with Linear support (both for primal and dual)} \\
%    & \text{and: hypergraphs with Linear support} \subset \text{ hypergraphs with Hered. Linear Delaunay graph}
%\end{aligned}
%\]}

\medskip \noindent \textbf{Comparison to~\cite{Huang25,Pal25}.}  Very recently,  P\'{a}lv\"{o}lgyi and Z\'{o}lomy~\cite{Pal25} and Huang et al.~\cite{Huang25} proved $(p,2)$-theorems for the class of primal incidence hypergraphs of points and families of non-piercing regions. The transversal size in~\cite{Huang25}, which improves over the results of~\cite{Pal25}, is $O(p^9)$.

As incidence hypergraphs of points and families of non-piercing regions admit a planar support (by the result of Raman and Ray~\cite{Raman2020}) and thus admit a hereditarily $3$-linear Delaunay graph, Corollary~\ref{cor:3} above implies the following \emph{linear} $(p,2)$-theorem for this class of hypergraphs. 
\begin{corollary}\label{cor:4}
There exists $c>0$ such that for any family $\F$ of non-piercing regions in the plane, the primal incidence hypergraph of points and regions in $\F$ admits a $(p,2)$-theorem with a transversal of size $cp$.  
\end{corollary}
Hence, our Theorem~\ref{thm:2} and Corollary~\ref{cor:3} significantly improve (from $O(p^9)$ to $O(p)$) and generalize (from non-piercing regions to hypergraphs that admit a hereditarily linear Delaunday graph) the results of~\cite{Huang25,Pal25}. We note that  Corollary~\ref{cor:4} (but not the more general Theorem~\ref{thm:2} and Corollary~\ref{cor:3}) can be obtained also by combining the results of~\cite{Huang25,Pal25} with a previous result of Dalal et al.~\cite{Dalal24}, as we show below in Proposition~\ref{obs:O(p)}.

\medskip \noindent \textbf{Comparison to~\cite{PyrgaR08}.} Pyrga and Ray~\cite{PyrgaR08} proved that the existence of a linear support\footnote{In fact, \cite{PyrgaR08} assumes a slightly stronger condition than admitting a linear support. In Appendix \ref{app} we show that their assumption can be weakened.} for the dual hypergraph $H^*$ implies the existence of a linear $\epsilon$-net for the hypergraph $H$.
We show that the assumption that the class $\famH_{\famG}^*$ has a linear support implies a linear lopsided Zarankiewicz theorem for the family $\famG$, with respect to containment $K_{t,2}$ (for all $t$).
\begin{proposition}\label{thm:supp->Zaran}
	Let $\alpha>0$. Let $G=(P \cup S, E)$ be a bipartite graph, $|P|=n,|S|=m$, such that the dual neighborhood hypergraph $H_G^*=(S,P)$ admits an $\alpha$-linear support. If $G$ is $K_{t,2}$-free then $|E|\leq n+\alpha (t-1)m$. 
\end{proposition}
 This shows that Theorem~\ref{thm:lopsided->epsnet} generalizes the result of Pyrga and Ray~\cite{PyrgaR08} from the class of hypergraphs whose dual hypergraph admits a linear support to the wider class of neighborhood hypergraphs of bipartite graphs that admit a linear lopsided Zarankiewicz theorem.

\medskip \noindent \textbf{Comparison to~\cite{ChanH25,HunterMTS25,KellerS24}.} Several works studied Zarankiewicz's problem for intersection graphs of two families of pseudo-disks. Keller and Smorodinsky~\cite{KellerS24} proved a linear upper bound with respect to containment of $K_{t,t}$ for all $t$, improving over an $O(n\log \log n)$ bound of Chan and Har-Peled~\cite{ChanH25}. Hunter et al.~\cite{HunterMTS25} obtained a linear bound in an alternative way. The proofs in these works are quite elaborate, and we are not aware of a way to simplify them significantly even in the most basic case of avoiding containment of $K_{2,2}$.  

Proposition~\ref{thm:supp->Zaran} gives a very easy proof of a linear bound for the lopsided Zarankiewicz problem with respect to containment of $K_{t,2}$ for bipartite graphs whose dual neighborhood hypergraph admits a linear support -- a class that includes intersection graphs of two families of pseudo-disks (as was shown above), as well as other families of hypergraphs. Thus, comparing our results in the $K_{2,2}$ setting with the results of~\cite{ChanH25,HunterMTS25,KellerS24} on Zarankiewicz's problem for intersection graphs of pseudo-disks, our results provide a much simpler proof that holds in a more general setting (e.g., including not only pseudo-disks but also non-piercing regions). 
While the `avoiding $K_{2,2}$' setting is somewhat limited, in some cases a linear bound in this setting can be leveraged to a linear bound for the entire lopsided Zarankiewicz's problem -- for example, in families of graphs that have the \emph{density Erd\H{o}s-Hajnal property}, as was shown in~\cite{HunterMTS25}. 
 
\subsection{A figure summarizing our results and open problems} Our work presents a landscape of sufficient conditions for the existence of linear $\epsilon$-net theorems
and $(p,2)$-theorems, with complex interrelations between them. There conditions are demonstrated in Figure~\ref{fig:fig1}.

As can be seen in the figure, some of the interrelations are still unknown and call for future research. We discuss these open problems in Section~\ref{sec:disc}.

	\begin{figure}[ht]
		\begin{center}
			\scalebox{1.22}{
				\includegraphics[width=0.8\textwidth]{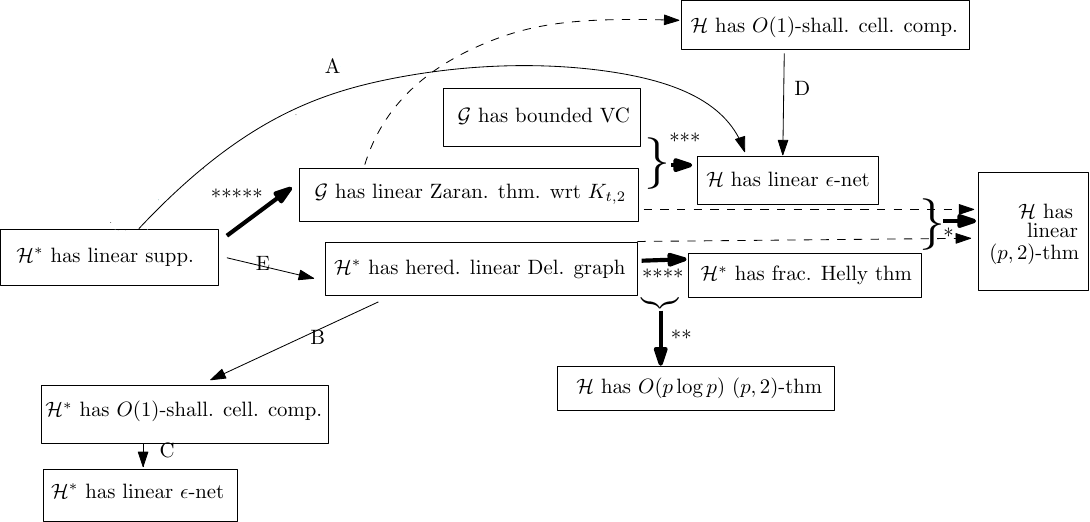}
			}
			\caption{The diagram depicts the relationships between the various problems discussed in this paper. Here, $\G$ is a class of bipartite graphs and $\famH=\famH_{\G}$ is the class of the corresponding primal neighborhood hypergraphs. Our main results are indicated by bold arrows. 
            The references to the text are as follows: * = Corollary \ref{cor:3}, ** = Theorem \ref{thm:2}, *** =Theorem \ref{thm:lopsided->epsnet}, **** = Proposition~\ref{thm:1}, ***** = Proposition~\ref{thm:supp->Zaran}.
            The normal-line arrows illustrate previously known relations. In particular, `A' is the main result of \cite{PyrgaR08}, `B' is mentioned in \cite{Raman2020} and is discussed in Section \ref{sec:disc}, `C' and `D' are proved in~\cite{ChanGKS12,MustafaDG18} and are discussed in Section \ref{sec:disc}, and `E' follows immediately from the definitions. The dashed arrows denote open problems that are discussed in Section \ref{sec:disc}.}
			\label{fig:fig1}
		\end{center}
	\end{figure}

\section{From Zarankiewicz's Problem to Linear $\epsilon$-Nets}

In this section we prove Theorem~\ref{thm:lopsided->epsnet} and Proposition~\ref{thm:supp->Zaran} that relate bounds for the lopsided Zarankiewicz problem with linear $\epsilon$-net theorems and admitting a planar support.

%The aforementioned Theorem \ref{thm:lopsided->epsnet} relates admitting a hereditarily optimal lopsided Zarankiewicz theorem to admitting a linear $\epsilon$-net.
%Let us recall its statement.

\medskip

\noindent \textbf{Theorem \ref{thm:lopsided->epsnet} - Restatement.}
	If a class $\G$ of bipartite graphs with bounded VC-dimension admits a hereditarily linear lopsided Zarankiewicz theorem (even only w.r.t.~containment of $K_{t,2}$ for any $t$), then for every $G \in \G$, the hypergraph $H_G$ admits an $\epsilon$-net of size $O(\frac{1}{\epsilon})$.

\begin{proof}
	Let $\epsilon>0$ and let $ 1 \leq c=c(\G)$ be a constant such that for every $G \in \G$ and for every finite induced graph $G'=G[P'\cup S']$, if $G'$ is $K_{t_1,t_2}$-free, then $\card {E(G')} \leq c(t_2|P'|+t_1|S'|)$. Let $\beta = \frac{1}{2c}$.
	
	Let $G=(P\cup S,E) \in \G$ with $|P|=n$, and let $\{s_1,\ldots,s_x\} \subset S$ be an (inclusion) maximal subset of $S$ with the following properties:
	\begin{itemize}
		\item $\forall 1 \leq i \leq x$, $\epsilon n \leq \deg_Gs_i<2 \epsilon n$;
			\item $\forall 1 \leq i<j \leq x$, $|N(s_i) \cap N(s_j)|< \beta \epsilon n$.
	\end{itemize}
Since $\G$ has bounded VC-dimension, for each $1 \leq i \leq x$ the hypergraph $H_{G[{N(s_i) \cup S}]}$ admits a $(\frac{1}{2} \beta)$-net $N_i$ of a constant size.

First, we claim that $N=\bigcup_{i=1}^x N_i$ intersects each hyperedge of $H_G$ whose size is between $\epsilon n$ and $2 \epsilon n$. Indeed, for each $s \in \{s_1,\ldots,s_x\}$, we have $N \cap N(s) \neq \emptyset$. On the other hand, by the maximality of $\{s_1,\ldots,s_x\} $, for each $s \in S \setminus \{s_1,\ldots,s_x\} $ with $\epsilon n \leq \deg_G s \leq 2\epsilon n$, there exists $1 \leq i \leq x$ such that $|N(s_i) \cap N(s)| \geq \beta \epsilon n$. But since $\beta \epsilon n \geq \frac{1}{2} \beta \deg_Gs_i = \frac{1}{2} \beta |N(s_i)|$, we have $N_i \cap N(s) \neq \emptyset$. Hence $N$ intersects each hyperedge of $H_G$ whose size is between $\epsilon n$ and $2 \epsilon n$. 

Second, we prove that $|N|=O(\frac{1}{\epsilon})$. Since each $N_i$ is of a constant size, we just need to prove that $x=O(\frac{1}{\epsilon})$. To see this, note that 
\begin{equation}\label{eq:1}
	\epsilon n x \leq |E(G [{P \cup \{s_1,\ldots,s_x\}}])| \leq c(2n + \beta \epsilon n x).
\end{equation}
The left inequality holds since $\deg_Gs_i=|N(s_i)| \geq \epsilon n$ for all $1 \leq i \leq x$. The right inequality holds since the graph $G [{P \cup \{s_1,\ldots,s_x\}}]$ is 
$K_{\beta \epsilon n,2}$-free by construction (every two neighborhoods of some $s_i$ and $s_j$ intersect in $< \beta \epsilon n$ vertices) and $\G$ admits a hereditarily linear lopsided Zarankiewicz theorem. By the choice of $\beta$, the inequality (\ref{eq:1}) implies $x\leq \frac{4c}{\epsilon}$.

\medskip

So far we obtained a piercing set of size $O(\frac{1}{\epsilon})$ only for the hyperedges of size between $\epsilon n $ and $2 \epsilon n$. In order to obtain a linear $\epsilon$-net for the whole hypergraph $H_G$, we apply the same argument recursively with $2 \epsilon, 4 \epsilon, 8 \epsilon \ldots$ to pierce the higher-order hyperedges, and the total size of the nets we form in each step is $O(\frac{1}{\epsilon}+\frac{1}{2 \epsilon}+\frac{1}{4 \epsilon}+\ldots)=O(\frac{1}{\epsilon})$ as asserted.
\end{proof}

The proof method of Theorem \ref{thm:lopsided->epsnet} is partially based on the strategy of~\cite{PyrgaR08}, where the existence of a linear support
%\footnote{Actually, the authors of \cite{PyrgaR08} state a more complicated assumption, but we show that their conditions can be weakened, see Appendix \ref{app}.} 
for the dual hypergraph is assumed instead of our Zarankiewicz-type assumption. We now prove Proposition~\ref{thm:supp->Zaran} which asserts that existence of a linear support implies a linear bound for the lopsided Zarankiewicz problem with respect to containment of $K_{t,2}$, for all $t$. Since the proof of Theorem~\ref{thm:lopsided->epsnet} uses the existence of a linear lopsided Zarankiewicz theorem only with respect to avoiding $K_{t,2 }$ (and not the general version where the forbidden subgraph is a general $K_{t_1,t_2 }$), this implies that Theorem~\ref{thm:lopsided->epsnet} generalizes the result of~\cite{PyrgaR08}.
%In \cite{PyrgaR08} it was proved that the existence of a planar support implies a linear $\epsilon$-net. Proposition \ref{thm:supp->Zaran} shows that the assumption of  Theorem \ref{thm:lopsided->epsnet} is weaker, as it follows easily from the existence of a planar support.

\medskip \noindent \textbf{Proposition \ref{thm:supp->Zaran} - Restatement.}
	Let $G=(P \cup S,E)$ be a bipartite graph, $|P|=n,|S|=m$, such that the dual hypergraph $H^*=(S,P)$ admits an $\alpha$-linear support. If $G$ is $K_{t,2}$-free then $|E|\leq n+\alpha (t-1)m$.

\begin{proof}
	Let $X_p$ denote the number of edges of the $\alpha$-linear support induced on the dual hyperedge $p^*=\{s \in S: \{p,s\} \in E\}$. Then since the restriction of the $\alpha$-linear support to $p^*$ is connected, we have $X_p \geq \deg_Gp-1$. It follows that $\Sigma_{p \in P} X_p \geq (\Sigma_{p \in P} \deg_Gp) -n =|E|-n$.
	
	On the other hand, $\Sigma_{p \in P} X_p \leq \alpha m (t-1)$ since the $\alpha$-linear support of $H^*=(S,P)$ has at most $\alpha m$ edges and each such edge participates in less than $t$ dual hyperedges $p^*$ by the $K_{t,2}$-freeness assumption. These two inequalities together imply the assertion. 
\end{proof}

\section{From a Hereditarily Linear Delaunay Graph to $(p,2)$-Theorems}

\subsection{A linear fractional Helly theorem}

Our starting point in this section is Proposition \ref{thm:1} -- a fractional Helly theorem for hypergraphs with a hereditarily linear Delaunay graph that seems to be of independent interest. 

Recall that for a hypergraph $H=(V,\E)$, two vertices $p,q \in V$ are called \emph{friends} if there exists some $e \in \E$ such that $p,q \in e$.

\medskip \noindent \textbf{Proposition~\ref{thm:1} - restatement.}
	Let $\famH$ be a class of hypergraphs with a hereditarily $c$-linear Delaunay graph, and let $0 < \alpha <1$. Then there exists $0<\beta = \beta (\alpha,c)<1$ such that for each $H=(V,\E) \in \famH$, if $\alpha \binom{n}{2}$ of the pairs in $\binom{V}{2}$ are friends, then there exists $e \in \E$ such that $|e|>\beta n$. Moreover, $\beta = \Omega(\alpha)$.

\medskip We call such a theorem a \emph{linear} fractional Helly theorem since it guarantees a linear dependence of $\beta$ on $\alpha$, which is clearly the `best' dependence one can hope for.

\medskip

To prove the proposition, we need the following corollary of a result of Ackerman et al.~\cite[Lemma 22]{AKP21}. 
\begin{lemma}\label{lem:friends}
	Let $H=(V,\E)$ be a hypergraph with hereditarily $c$-linear Delaunay graph. Let $k=\max_{e \in \E}|e|$ and let $X=\{  \{ p,q\} : p,q \in V \mbox{ are friends} \}$. Then $|X|=O(nk)$.
\end{lemma}

\begin{proof}[Proof-sketch]
The proof is a standard probabilistic argument. Form a random subset $V' \subset V$ by picking each vertex of $V$ independently with probability $x$. Let $G'$ be the Delaunay graph of $H{\mid}_{V'}$. By passing to the expectation, we obtain $(1 - x)^{k-2} x^2 |X| \leq \mathbb{E} ( | E(G')|) \leq c x n$, and for $x = \frac{1}{k-2}$ this yields $|X|=O(nk)$ as asserted.
\end{proof}

\medskip \noindent Now we are ready to prove Proposition \ref{thm:1}. 
%Let us recall its statement.

\begin{proof}[Proof of Proposition \ref{thm:1}]
	Let $H=(V,\E)\in \famH$. Let $k=\max_{e \in \E}|e|$. Let $X=\{ \{p,q\} : p,q \in V \mbox{ are friends}   \}$. Then by assumption, $|X| \geq \alpha \binom{n}{2}$ and by Lemma \ref{lem:friends}, $|X|=O(nk)$. It follows that $k=\Omega (\alpha n)$, and the assertion follows.
\end{proof}

\subsection{An almost linear $(p,2)$-theorem}

Recall that our $(p,2)$-theorem asserts the following.

\medskip \noindent \textbf{Theorem~\ref{thm:2} - Restatement.}
Let $\famH $ be a class of hypergraphs such that $\famH^*$ admits a hereditarily linear Delaunay graph. Then $\famH$ admits a $(p,2)$-theorem with a transversal of size $O(p \log p)$.
	%Let $\F$ be a family of $n$ regions in $\Re^d$ and let $H$ be the associated dual hypergraph. Assume that $H$ admits a hereditarily linear Delaunay graph. If $\F$ satisfies the $(p,2)$-property then $\F$ admits an $O(p \log p)$-sized transversal.

\medskip

Our proof of Theorem \ref{thm:2} follows the strategy of the Alon-Kleitman proof of the $(p,q)$-theorem \cite{Alon1992}. Recall that the Alon-Kleitman proof has four steps:
\begin{enumerate}
    \item Proving that the number of intersecting $q$-tuples of hyperedges in the hypergraph is `large';

    \item Using a fractional Helly theorem to deduce the existence of a `deep' vertex -- i.e., a vertex that belongs to many hyperedges;
    
    \item Using the linear programming (LP) duality lemma to deduce the existence of a finite weighted set $Q$ of vertices such that every hyperedge contains a subset of $Q$ of a `large' weight;
    
    \item Applying the \emph{weak $\epsilon$-net theorem}~\cite{AlonBFK92} to the primal hypergraph induced on the vertex set $Q$ to deduce the existence of a small-sized trasversal.
\end{enumerate}

The analogue of the first step in our proof is showing that a hypergraph that satisfies the assumption of Theorem~\ref{thm:2} contains `many' intersecting pairs of edges. This is obtained using the following corollary of Tur\'{a}n's theorem.
\begin{lemma}\label{lem:Turan}
    Let $G=(V,E)$ ($|V|=n$) be a graph such that every subset $V' \subset V$ with $|V'|=p$ contains some edge of $E(G)$. Then $|E(G)| \geq \frac{n^2}{2(p-1)}-\frac{n}{2}$.
\end{lemma}

\begin{proof}
Consider the complement graph $\Bar{G}$. This graph does not contain $K_p$, hence by Tur\'{a}n's theorem $|E(\Bar{G})| \leq (1-\frac{1}{p-1})\frac{n^2}{2}$, and it follows that $|E(G)|=\binom{n}{2}-|E(\Bar{G})| \geq \frac{n^2}{2(p-1)}-\frac{n}{2}$.
\end{proof}

The analogue of the second step is applying the fractional Helly theorem (i.e., Proposition~\ref{thm:1} above). For the third step, we recall the LP-duality lemma.

%Before proving our $(p,2)$-theorem (Theorem \ref{thm:2}), let us mention two ingredients from the Alon-Kleitman proof of the $(p,q)$-theorem \cite{Alon1992}. The first is a linear-programming duality lemma.

\medskip \noindent \textbf{LP duality lemma.} %Let $0 < \alpha <1$ and let $\F$ be a family of sets such that for any multiset $\F'$ of elements of $\F$, there exists a point stabbing at least $\alpha |\F'|$ elements of $\F'$. Then there exists a finite multiset $Q$ such that each member of $\F$ contains at least $\alpha |Q|$ points of $Q$. 
Let $0 < \alpha <1$ and let $(P,S)$ be a set system such that for any multiset $S'$ of elements of $S$, there exists $p \in P$ contained in at least $\alpha |S'|$ elements of $S'$. Then there exists a finite multiset $Q \subset P$ such that each member of $S$ contains at least $\alpha |Q|$ elements of $Q$. 

\medskip

In the fourth step, we replace the weak $\epsilon$-net theorem used in the Alon-Kleitman proof with application of the aforementioned \emph{$\epsilon$-net theorem} of Haussler and Welzl~\cite{HausslerW87} (improved by Koml\'{o}s, Pach and Woeginger~\cite{Komlos1992b}).  
%Recall that for $\epsilon>0$, an $\epsilon$-net of a hypergraph $H=(V,\E)$ is a set of vertices $S \subset V$ such that for any $e \in \E$ with $|e|>\epsilon|V|$, we have $S \cap e \neq \emptyset$.

\medskip \noindent \textbf{The $\epsilon$-net theorem (\cite{HausslerW87,Komlos1992b}).} Let $\epsilon>0$. Any hypergraph $H=(V,\E)$ with VC-dimension $d$ admits an $\epsilon$-net of size at most $O \left(\frac{d}{\epsilon} \log \frac{1}{\epsilon} \right)$.

%\medskip Like in the Alon-Kleitman proof, we deduce the $(p,2)$-theorem by combining the existence of a deep point that follows from Theorem \ref{thm:1}, together with the LP-duality lemma and the $\epsilon$-net theorem. 

\medskip In the proof of  Theorem~\ref{thm:2} and in the sequel, we use the following standard definition. A hypergraph $H = (V,\E)$ \emph{admits the $(p,2)$-property} if among any $p$ hyperedges in $\E$, some two intersect. 

\begin{proof}[Proof of Theorem \ref{thm:2}]
Let $H=(P,S) \in \famH$ with $|S|=n$.
Assume that $H$ admits the $(p,2)$-property. 
	Let $G$ be a graph whose vertex set is $S$ and $(s_1,s_2)$ is an edge if $s_1 \cap s_2 \neq \emptyset$. Then by the $(p,2)$-property of $H$, the condition of Lemma \ref{lem:Turan} is satisfied, and the lemma implies that at least $\Omega(\frac{n^2}{p})$ pairs of hyperedges in $S$ intersect. Then by applying Theorem \ref{thm:1} to $H^*$, we get that there exists $p \in P$ that is contained in $\Omega(\frac{n}{p})$ hyperedges in $S$. For every multiset $S'$ of elements of $S$, the hypergraph $H'=(P,S')$ satisfies the $(p,2)$-property as well, and hence the LP-duality lemma implies the existence of a finite multiset $Q$ of elements of $P$, such that each hyperedge in $S$ contains at least $\frac{|Q|}{p}$ points from $Q$.
	
	$H^*$ has a bounded VC-dimension, being a hypergraph with a hereditarily linear Delaunay graph, see \cite[Theorem 6]{AKP21}. Thus, the primal hypergraph $H=(P,S)$ has a bounded VC-dimension too, and so does $H(Q,S)$.
	Applying the $\epsilon$-net theorem to the hypergraph $H(Q,S)$, with $\epsilon=\frac{1}{p}$, we obtain a subset of $Q$ of size $O(p \log p)$ which is a transversal for $H=(P,S)$. This completes the proof.      
\end{proof} 

\subsection{A linear $(p,2)$-theorem and an application to non-piercing regions}

If $\famH$ admits a linear-sized $\epsilon$-net, then Theorem \ref{thm:2} can be strengthened to a linear $(p,2)$-theorem. 
%Recall the statements of Corollaries \ref{cor:3} and~\ref{cor:4} -- the linear $(p,2)$-theorem for non-piercing regions.
This is stated in Corollaries \ref{cor:3} and~\ref{cor:4}, which we prove together.

\medskip

\noindent \textbf{Corollaries \ref{cor:3} and \ref{cor:4} -- Restatement.}
%Under the conditions of Theorem \ref{thm:2}, 
Let $\famH $ be a class of hypergraphs such that $\famH^*$ admits a hereditarily linear Delaunay graph and $\famH$ admits a linear-sized $\epsilon$-net. Then $\famH$ admits a $(p,2)$-theorem with a transveral of size $O(p)$. 

In particular, any family $\F$ of non-piercing regions in the plane that satisfies the $(p,2)$-property, admits a transversal of size $O(p)$.

\begin{proof}[Proof of Corollaries \ref{cor:3} and \ref{cor:4}]
The first assertion follows immediately from the proof of Theorem \ref{thm:2}, as in the fourth step, the assumption that $\famH$ admits a linear-sized $\epsilon$-net can be used instead of the $\epsilon$-net theorem to yield a transversal of size $O(p)$.

In the special case of non-piercing regions, 
%observe that in this case 
the dual hypergraph of $\F$ admits a hereditarily planar, and hence $3$-linear, Delaunay graph. Indeed, the dual hypergraph $H^*=(\F,\Re^2)$ of $\F$ admits a planar support (as was proved in \cite{Raman2020} and later reproved in \cite{Dalal24}). Therefore, 
%by Observation \ref{obs:supp->del}, 
the Delaunay graph of $H^*$ is clearly a subgraph of a planar graph and hence is 3-linear. In addition, the primal hypergraph $H$ of a family of non-piercing regions admits an $\epsilon$-net of size $O(\frac{1}{\epsilon})$, by \cite[Theorem 24]{AJKSY22} (based on \cite{PyrgaR08}). Therefore, the proof of Theorem \ref{thm:2} can be applied with the linear sized $\epsilon$-net, and the assertion of the corollary follows.  
\end{proof}

We note that Corollary \ref{cor:4} (but not Theorem~\ref{thm:2} and Corollary~\ref{cor:3}) follows also from a combination of the polynomial $(p,2)$-theorem for families of non-piercing regions proved in \cite{Huang25,Pal25} with the following lemma of Dalal et al. \cite[using a result of \cite{Pinchasi2014}]{Dalal24}.

\begin{lemma}\cite[Theorem 3]{Dalal24}\label{lem:157}
	Each family $\F$ of non-piercing regions contains some region $B \in \F$ such that the family $\{  A: A \in \F, A \cap B \neq \emptyset\}$ has the (157,2)-property.
\end{lemma}

\begin{proposition}
    \label{obs:O(p)}
	Any family $\F$ of non-piercing regions that admits the $(p,2)$-property, has a transversal of size $O(p)$.
\end{proposition}

\begin{proof}[Proof of Proposition \ref{obs:O(p)}]
	By Lemma \ref{lem:157}, there exists $B \in \F$ such that the set $\F_1=\{ A: A \in \F , A \cap B \neq \emptyset\}$ satisfies the (157,2)-property. By~\cite{Huang25} or~\cite{Pal25}, $\F_1$ can be pierced by $C$ points, for some fixed constant $C$. Removing $\F_1$, the remaining family $\F \setminus \F_1$, satisfies the $(p-1,2)$-property. We continue in the same manner at most $p-1$ times, and obtain a transversal of $\F$ that contains at most $C(p-1)$ points. 
\end{proof}

\section{Discussion and Open Problems}\label{sec:disc}

In this paper we discussed several sufficient conditions for the existence of linear $\epsilon$-net theorems and linear $(p,2)$-theorems. While the relations between some of these conditions are clear or were established in previous works or in our paper, many of the relations are unknown and call for future work. Thus, we conclude the paper with a list of open problems. 

\begin{problem}
    Does the existence of a hereditarily linear lopsided Zarankiewicz theorem for some class $\G$ of graphs imply the existence of a linear $(p,2)$-theorem for $\famH_{\G}$?
\end{problem}
The motivation behind this problem is as follows. Pyrga and Ray~\cite{PyrgaR08} showed that the existence of a linear support for $\famH_{\G}^*$ implies the existence of a linear $\epsilon$-net for $\famH_{\G}$, and we showed in Theorem~\ref{thm:lopsided->epsnet} and Proposition~\ref{thm:supp->Zaran} that the assumption can be weakened to assuming the existence of a hereditarily linear lopsided Zarankiewicz theorem for $\G$. The problem asks whether a similar weakening of the assumption can be made in a weaker version of our Theorem \ref{thm:2} which states that the existence of a linear support for $\famH_{\G}^*$ implies the existence of a linear $(p,2)$-theorem for $\famH_{\G}$.

\begin{problem}
    Is it true that any $K_{t,t}$-free bipartite intersection graph of two families of non-piercing regions in the plane has at most $O(tn)$ edges?
\end{problem}
The motivation behind this problem is as follows. Proposition \ref{thm:supp->Zaran} implies, in particular, the existence of a linear lopsided Zarankiewicz theorem with respect to containment of $K_{t,2}$ for the intersection graph of two families of non-piercing regions. (The existence of a planar support for the associated dual neighborhood hypergraph was proved in \cite{Raman2020}). In particular, this implies that the number of edges in a $K_{2,2}$-free bipartite intersection graph of two families of $n$ non-piercing regions is $O(n)$. 
By the leveraging technique introduced by Hunter et al.~\cite[Corollary 1.10]{HunterMTS25}, this allows deducing a bound of $O(tn)$ for the Zarankiewicz problem with respect to containment of $K_{t,t}$ if such a bipartite intersection graph admits the \emph{density Erd\H{o}s-Hajnal-property}. Can one prove that this property indeed holds? 
%If yes, this will provide a new interesting upper bound for Zarankiewicz's problem.

%\begin{enumerate}
%	\item Does the existence of a hereditarily optimal lopsided Zarankiewicz theorem for some class $\G$ implies the existence of a linear $(p,2)$-theorem for $\famH_{\G}$? The motivation is as follows. Pyrga and Ray \cite{PyrgaR08} showed that (linear support for $\famH^*$) $\to$ (linear $\epsilon$-net for $\famH$), and we showed in Theorem~\ref{thm:lopsided->epsnet} and Proposition~\ref{thm:supp->Zaran} above that one can insert (Zarankiewicz) in between. We proved in Theorem \ref{thm:2} above that (planar support for $\famH^*$) $\to$ (p,2) theorem for $\famH$, so again one can ask whether Zarankiewicz can be inserted in between.
	
%	\item Proposition \ref{thm:supp->Zaran} implies, in particular, the existence of a Zarankiewicz-type result for the intersection graph of two families of non-piercing regions. (The existence of a planar support for the associated dual neighborhood hypergraph was proved in \cite{Raman2020}.) A special case of this result is that the number of edges in an intersection graph of two families of $n$ non-piercing regions which is $K_{2,2}$-free is linear in $n$. By a recent result of \cite[Corollary 1.10]{HunterMTS25}, If such an intersection graph admits the density Erd\H{o}s-Hajnal-property, then a general Zarankiewicz result, linear in both $n$ and $t$, follows. Can one prove the missing density EH-property?
%\end{enumerate}

\medskip Before presenting the next open problem we need one more definition:
\begin{definition}
For a function $\varphi: \mathbb{N} \to \mathbb{N}$, we say that a hypergraph $H=(V,\E)$ has \emph{shallow cell complexity} $\varphi(m)$ if there exists $c(H)$ such that for any $m \in \mathbb{N}$ and any $l \leq m$, 
$\max_{S \subset V, |S|=m}| \{S \cap e: e \in \E\ , |S \cap e|\leq l\} \leq m \varphi(m) \cdot l^{c(H)}$ (see~\cite{ChanGKS12}).
\end{definition}
We say that a class $\famH$ of hypergraphs has shallow cell complexity $\varphi(m)$ if every $H \in \famH$ has shallow cell complexity $\varphi(m)$. 
\begin{problem}\label{Q:Shallow}
Does the existence of a hereditarily linear lopsided Zarankiewicz theorem (w.r.t. $K_{t,2}$) for a class $\G$ with a bounded VC-dimension imply that $\famH_{\G}$ has shallow cell complexity of $O(1)$? 
\end{problem}
The motivation behind this problem is as follows. Mustafa, Dutta and Ghosh~\cite{MustafaDG18}, following Chan et al.~\cite{ChanGKS12}, proved that having shallow cell complexity of $O(1)$ implies admitting a linear-sized $\epsilon$-net, and showed that all known results asserting the existence of linear $\epsilon$-nets follow from this proof strategy. In particular, this is the case for the statement that any family $\famH$ that has a linear support (or even `only' has a hereditarily linear Delaunay graph) admits a linear-sized $\epsilon$-net, as by an argument of Raman and Ray~\cite{Raman2020}, such hypergraphs have shallow cell complexity of $O(1)$.  
Hence, if the answer to Problem~\ref{Q:Shallow} is negative, then our Theorem \ref{thm:lopsided->epsnet} gives the first sufficient condition for the existence of a linear-sized $\epsilon$-net that does not follow from having shallow cell complexity of $O(1)$.

\begin{problem}\label{Q:epsnet-todual}
Does the existence of a linear $\epsilon$-net for a hypergraph $H$ imply the existence of a linear $\epsilon$-net for its dual hypergraph $H^*$?    
\end{problem}
A positive answer would imply that the assertion of Theorem~\ref{thm:2} can be strengthened to guaranteeing a \emph{linear} $(p,2)$-theorem. Indeed, Corollary~\ref{cor:3} asserts that the size of the transversal in the $(p,2)$-theorem guaranteed by Theorem~\ref{thm:2} can be reduced to $O(p)$ if in addition to the assumption that $\famH^*$ admits a hereditarily linear Delaunay graph, one assumes that $\famH$ admits a linear $\epsilon$-net theorem. As was written in the previous parargaph, the `hereditarily linear Delaunay' assumption implies that $\famH^*$ admits a linear $\epsilon$-net theorem. Hence, a positive answer to Problem~\ref{Q:epsnet-todual} would imply that the `hereditarily linear Delaunay' assumption implies the `linear $\epsilon$-net' assumption and leads to a linear $(p,2)$-theorem.

A weaker statement is that Theorem~\ref{thm:2} can be strengthened into a linear $(p,2)$-theorem (i.e., that the existence of a hereditarily linear Delaunay graph for $\famH^*$ implies a linear $(p,2)$-theorem for $\famH$), even though the answer to Problem~\ref{Q:epsnet-todual} is generally negative.\footnote{We note that the existence of a hereditarily linear Delaunay graph for $H$ does not imply the same property for $H^*$, as can be easily seen by considering the hypergraph $H_G$, where $G$ is a 1-subdivision of a complete graph.}  See the diagram in Figure \ref{fig:fig1}.

	\bibliographystyle{plain}
	\bibliography{references}
	
\appendix
\section{Linear support implies bounded VC-dimension}
\label{app}
In \cite[Theorem~4]{PyrgaR08}, Pyrga and Ray showed that if $H^*$ has a planar support (in fact, a slightly weaker condition suffices) and $H$ has a bounded VC-dimension (in fact, a slightly weaker condition suffices), then $H$ admits a linear-sized $\epsilon$-net. In this appendix we show that the second assumption is, in fact, unnecessary: the existence of a linear support for $H^*$ already implies that $H^*$ has a bounded VC-dimension, and consequently, $H$ has a bounded VC-dimension as well.
%it is a standard fact that bounded VC-dimension of $H^*$ entails bounded VC-dimension of $H$.

\medskip

Formally, Pyrga and Ray \cite{PyrgaR08} proved the following:

\begin{theorem} \cite[Theorem 4]{PyrgaR08}\label{thm:app}
Any hypergraph $H=(P,S)$ satisfying the following two conditions, admits an $\epsilon$-net of size $O(\frac{1}{\epsilon})$.
\begin{enumerate}
    \item $H$ has bounded VC-dimension\footnote{The original condition in \cite[Theorem 4]{PyrgaR08} was somewhat weaker: For any $0 < \epsilon<1$ and for any $P' \subset P$, the hypergraph $H|_{P'}$ admits an $\epsilon$-net whose size depends only on $\epsilon$.}.
    \item There exist constants $\alpha>0,\beta \geq 0 $ and $\tau >0$ s.t. for any $I \subset S$ there is a graph $G_I=(I,E_I)$ with $|E_I| \leq \beta |I|$ such that every dual hyperedge $p^*$ of size at least $k$, contains at least $\alpha k - \tau$ edges of $G_I$ (where we identify an edge of $G_I$ with the set of its two endpoints). 
\end{enumerate}
\end{theorem}

The following claim shows that Condition (1) follows from Condition (2).
\begin{claim}\label{cl:app}
Assume that $H=(P,S)$ is a hypergraph that satisfies the second condition of Theorem \ref{thm:app}. Then $VC$-$\dim(H^*)<m$ where $m$ is a large integer that satisfies 
\begin{itemize}
    \item $\frac{\tau+1}{\alpha}<m$;
    \item $\beta m <\frac{m^2}{2(\frac{\tau+1}{\alpha}-1)}-\frac{m}{2}.$
\end{itemize}
\end{claim}
\begin{proof}
    Assume to the contrary that there exists $I \subset S,|I|=m$ such that $$I \cap \{p^*: p^* \mbox{ is a hyperedge of } H^*\}=2^{I}.$$ In particular, every $\frac{\tau+1}{\alpha}$-subset of $I$ is obtained by an intersection of $I$ with some hyperedge $p^*$ of $H^*$. By our assumption, $I \cap p^*$ contains the two endpoints of some edge of $G_I $. Therefore the graph $G_I$ has the $\frac{\tau+1}{\alpha}$-property, namely, every $\frac{\tau+1}{\alpha}$-subset of its vertices contains some edge of it. By Lemma \ref{lem:Turan}, this implies that $|E(G_I)| \geq \frac{m^2}{2(\frac{\tau+1}{\alpha}-1)}-\frac{m}{2}>\beta m$, in contradiction to $G_I$ being $\beta$-linear.  
\end{proof}	

\end{document}